\documentclass[11pt]{amsart}

    \usepackage{geometry}

\newgeometry{vmargin={30mm}, hmargin={25mm,20mm}}  
\usepackage{amssymb,amsfonts}
\usepackage[all,arc]{xy}
\usepackage{enumerate}
\usepackage{mathrsfs}
\usepackage{graphicx}
\usepackage{caption}
\usepackage{subcaption}
\usepackage{ textcomp }
\usepackage{ upgreek }
\usepackage{tikz}
\usepackage{leftidx}
\usepackage{enumitem}
\usepackage{nameref}
\usepackage[english]{babel}
\usetikzlibrary{calc}

\newtheorem{thm}{Theorem}[section]
\newtheorem{cor}[thm]{Corollary}

\newtheorem{lem}[thm]{Lemma}

\newtheorem{defn}[thm]{Definition}

\theoremstyle{definition}

\def\<{\langle}
\def\>{\rangle}

\theoremstyle{remark}

\newenvironment{pfof}[1]{\par\medskip\noindent\textit{Proof of #1.}~}{\hfill $\square$\par\medskip}

\makeatletter
\let\c@equation\c@thm
\makeatother
\numberwithin{equation}{section}

\date{\today}

\begin{document}
	\author[Chinyere]{I. Chinyere}
	\address{ Ihechukwu Chinyere\\
		Department of Mathematics\\ Michael Okpara University of Agriculture\\ Umudike\\
		P.M.B 7267 Umuahia\\ Abia State\\ Nigeria.}
	\email{ihechukwu@aims.ac.za}
	\address{
		African Institute for Mathematical Sciences (AIMS), Cameroon}
	\email{ihechukwu.chinyere@aims-cameroon.org}

	\title{Structure of words with short $2$-length in a free product of groups.}

	\begin{abstract}
			Howie and Duncan observed that a word in a free product with length at least two and which is not a proper power can be decomposed as a product of two cyclic subwords each of which is uniquely positioned. Using this property, they proved various important results about one-relator product of groups. In this paper, we show that similar results hold in a more general setting where we allow elements of order two.
	\end{abstract}

	\maketitle
	
	\section{INTRODUCTION}
	Let $R$ be a cyclically reduced word which is not a proper power and has length at least two in the free group $F = F(X)$. In  \cite{We1}, Weinbaum showed that some cyclic conjugate of $R$ has a decomposition of the form $UV$, where $U$ and $V$ are non-empty cyclic subwords of $R$, each of which is uniquely positioned in $R$ i.e. occurs exactly once as a cyclic subword of $R$. Weinbaum also conjectured that $U$ and $V$ can be chosen  so that neither is a cyclic subword of $R^{-1}$. A stronger version of his conjecture was proved by Duncan and Howie \cite{dh}. In this paper, a cyclic subword is uniquely positioned if it is non-empty, occur exactly once as a subword of $R$ and does not occurs as a subword of $R^{-1}$.
	
	\medskip
	From now on $R$ is a word in the free product of groups $G_1$ and $G_2$, which is not a proper power and has length at least two. Before we can continue, we need to  define the notion of $n$-\textit{length} of a word. We do this in the special case when $n=2$ and the word is $R$, but of course the definition can be generalised for any $n>1$ and any word in a group.

	\medskip
	For each element $a$ of order $2$ involved in $R$, let $D(a)$ denote its number of occurrence in $R$. In other words suppose $R$ has free product length of $2k$ for some $k>0$. Then, $R$ has an expression of the form
	$$R=\prod_{i=1}^ka_{i}b_{i},$$
	with $a_i\in G_1$ and $b_i\in G_2$. If $a^2=1$, then we define $D(a)$ to be the cardinality of the set $\{i\in \{1,2,\cdots,k\}~|~a_i=a\}$.
	Denote by $\textbf{S}_{R}$ the symmetrized closure of $R$ in $G_1*G_2$ i.e. the smallest subset of $G_1*G_2$ containing $R$ which is closed under cyclic permutations and inversion. Since $D(a)$ is unchanged by replacing $R$ with any other element in $\textbf{S}_{R}$, we make the following definition.
	\begin{defn}
		The $2$-\textit{length} of $\textbf{S}_{R}$, denoted by $D_2(\textbf{S}_{R})$, is the maximum $D(a)$, such that $a$ is a letter of order $2$ involved in $R$.
	\end{defn}
    In this paper, we will be mostly concerned with the element $R'$  in $\textbf{S}_{R}$ of the form
	$$R'= \prod_{i= 1}^{D_2(\textbf{S}_{R})}aM_i,$$
	with $D(a) = D_2(\textbf{S}_{R})$ and $M_i\in G_1*G_2$. It follows that each $M_i$ has odd length (as a reduced but not cyclically reduced word in the free product) and does not contain any letter equal to $a$. When we use the  notation `$`=$" for words, it will mean identical equality. We will use $\ell()$ to mean then length of a reduced free product word which is not necessarily cyclically reduced. 
	
	\medskip
	As mentioned in the abstract, the authors of  \cite{dh} observed that in the case when $D_2(\textbf{S}_{R})=0$, the word $R$ can be decomposed as a product of two uniquely positioned subwords. Using that they showed that every minimal picture over a one-relator product with relator $R^3$ satisfies  $C(6)$, from which important results about the group were proved. In this paper we work in a more general setting where $D_2(\textbf{S}_{R})\leq 2$. It is no longer always possible that $R$ has a decomposition into uniquely positioned subwords. However, we can show that $R$ has a certain structure which allows us to obtain similar results.
This idea is captured in the following theorem.
	\begin{thm}\label{the1}
	Let $R$ be a word in a free product of length at least $2$ and which is not a proper power. Suppose also that $D_2(\textbf{S}_{R})\leq 2$.
	Then either a cyclic conjugate of $R$ has a decomposition of the form $UV$ such that $U$ and $V$ are uniquely positioned or one of the following holds:
	\begin{itemize}
		\item[(a)] $D_2(\textbf{S}_{R})=1$ and $R$ has a cyclic conjugate of the form $aXbX^{-1}$ or $aM$, where $a,b$ are letters of order $2$ and $M$ does not involve any letter of order $2$.
		\item[(b)] $D_2(\textbf{S}_{R})=2$ and $R$ has a cyclic conjugate of the form $aXbX^{-1}$ where $a$ is a letter of order $2$.
	\end{itemize} 
\end{thm}
 Note that in Theorem \ref{the1}, the requirement that $D_2(\textbf{S}_{R})\leq 2$ is optimal in the sense that there is no hope to obtain such result when $D_2(\textbf{S}_{R})> 2$. To see why this is true, consider the word $S=\prod_{i=1}^{n}ab_i$, with $a\in G_1$ and $b_i\in G_2$, $i=1,2,\cdots,n$. Suppose that  $b_i\neq b_j$ for $i\neq j$ and $a^2=b_i^2=1$ for $i=1,2,\cdots, n$. It is easy to verify that $D_2(\textbf{S}_{R})=n$ and Theorem \ref{the1} fails for $n>2$. In other words, $S$ does not have a decomposition into two uniquely positioned subwords, nor does it have a decomposition of the form $xXyX^{-1}$ such that $x^2=1$.
 
 \medskip
 Further analysis of the structure of $R$ leads us to the following theorem which is our main result in this paper.
 
 \begin{thm}\label{the2}
 Let $R$ satisfy the conditions of Theorem \ref{the1}. Then either any minimal picture over $G$ satisfies  $C(6)$ or $R$ has the form (up to cyclic conjugacy) $aXbX^{-1}$ with $a^2=1\neq b^2$.
 \end{thm}
 
 The rest of the paper is arranged as follows. We begin in Section $2$ by providing some literature on related results. We also recall only the basic ideas about pictures. In Section $3$ we prove a number of Lemmas about word combinatorics and pictures. In particular we deduce Theorem \ref{the1}. Furthermore, these Lemmas are applied in Section $4$ to prove the main result and deduce a number of applications.

	\section{PRELIMINARIES}
	
Let $G_1$ and $G_2$ be nontrivial groups and $w\in G_1*G_2$ a cyclically reduced word. Let $G$ be the quotient of the free product $G_1*G_2$ by the normal closure of $w$, denoted $N(w)$. Then $G$ is called a one-relator product and denoted by
$$G= (G_1*G_2)/N(w).$$
We refer to $G_1,G_2$ as the factors of $G$, and $w$ as the relator. For us, $w = R^m$ such that $R$ is a cyclically reduced word which is not proper power and  $m\geq 3$. If $m\geq 4$, a number of results were proved in \cite{jh1,jh2,jh3}, about $G$.
These results were also proved in \cite{dh} when $m= 3$ but not without the extra condition that $R$ involves no letter of order $2$. We also mention that the case when $m= 2$ is largely open. For partial results in this case see \cite{fhr,Ih1,Ih2}. The aim of this paper is to extend the result in \cite{dh} by allowing letters of order $2$ in $R$. Also we require results about pictures over $G$, in particular the fact that $R^m$ satisfies the small cancellation condition $C(2m)$ when $R$ has a certain form. Pictures can be seen as duals of van Kampen diagrams and have been widely used to prove results about one-relator groups and one-relator products. Below, we recall only basic concepts on pictures over a one-relator product as given in \cite{Ih1}. For more details, the reader can see \cite{jh1,jh2,jh3,dh,Ih2}.

\subsection{PICTURES}\label{sec4}
A \textit{picture} $\Gamma$ over $G$ on an oriented surface $\Sigma$ is made up of the following data:
\begin{itemize}
	\item a finite collection of pairwise 
	disjoint 
	closed discs in the interior of $\Sigma$ called \textit{vertices};
	
	\item a finite collection of disjoint closed arcs called \textit{edges}, each of which is either:  a simple closed arc in the interior of $\Sigma$ meeting no vertex of $\Gamma$, a simple arc joining two vertices (possibly same one) on $\Gamma$, a simple arc joining a vertex  to the boundary $\partial \Sigma$ of $\Sigma$, a simple arc joining $\partial \Sigma$ to $\partial \Sigma$;

	\item a collection of \textit{labels} (i.e words in $G_1\cup  G_2$), one for each corner of each \textit{region} (i.e connected component of the complement in $\Sigma$ of  the 
	union  of vertices and arcs of $\Gamma$) at a vertex and one along each component of the intersection of the region with $\partial \Sigma$. For each vertex, the label around it spells out the word $R^{\pm m}$ (up to cyclic permutation) in the clockwise order as a cyclically reduced word in $G_1 * G_2$. We call a vertex \textit{positive} or \textit{negative} depending on whether the label around it is $R^{ m}$ or $R^{-m}$ respectively.
\end{itemize}

For us $\Sigma$ will either be the $2$-sphere $S^2$ or $2$-disc $D^2$. A picture on $\Sigma$ is called \textit{spherical} if either $\Sigma=S^2$ or $\Sigma=D^2$ but with no arcs connected to $\partial {D^2}$. If 
$\Gamma$ is not spherical, $\partial {D^2}$ is one of the boundary components of a non-simply connected region (provided, of course, that $\Gamma$ contains at least one vertex or arc), which is called the \textit{exterior}. All other regions are called \textit{interior}.

\medskip
We shall be interested mainly in \textit{connected} pictures.
A picture is \textit{connected} if the union of its vertices and arcs is connected. In particular, no arc of a connected picture is a closed arc or joins two points of $\partial \Sigma$, unless the picture consists only of that arc.  In a connected picture, all  interior
regions $\bigtriangleup$ of $\Gamma$ are simply-connected, i.e topological discs. Just as in the case of vertices, the label around each region -- read {\em anticlockwise} -- gives a word which in a connected picture is required to be trivial in $G_1$ or $G_2$. Hence it makes sense to talk of $G_1-$regions or $G_2-$regions. Each arc is required to separate a $G_1-$region from a $G_2-$region. This is compatible with the alignment of regions around a vertex, where the labels spell a cyclically reduced word, so must come alternately from $G_1$ and $G_2$. 

\medskip
A vertex is called \textit{exterior} if it is possible to join it to the \textit{exterior} region by some arc without intersecting any  arc of $\Gamma$, and \textit{interior} otherwise. For simplicity we will indeed assume from this point that our $\Sigma$ is either $S^2$ or $D^2$. It follows that reading the label round any \textit{interior} region spells a word which is trivial in $G_1$ or $G_2$. The \textit{boundary label} of $\Gamma$ on $D^2$ is a word obtained by reading the \textit{labels} on $\partial D^2$ in an \textit{anticlockwise} direction. This word (which we may be assumed to cyclically reduced in $G_1 * G_2$) represents an identity element in $G$. In the case where $\Gamma$ is spherical, the \textit{boundary label} is an element in $G_1$ or $G_2$ determined by other labels in the \textit{exterior} region. 

\medskip
Two distinct vertices of a picture are said to  \emph{cancel} along an arc $e$ if they are joined by $e$ and if their labels, read from the endpoints of $e$, are mutually inverse
words in $G_1 * G_2$. Such vertices can be removed from a picture via a sequence of \textit{bridge moves} (see Figure \ref{bridge} and  \cite{dh} for more details), followed by deletion of a \textit{dipole} without changing the boundary label. A \textit{dipole} is a connected spherical sub-picture that contains precisely two vertices, does not meet $\partial \Sigma$, and such that none of its interior regions contain other components of $\Gamma$. This gives an alternative picture with the same boundary label and two fewer vertices. 

\begin{figure}[h!]
	\includegraphics[scale=0.5]{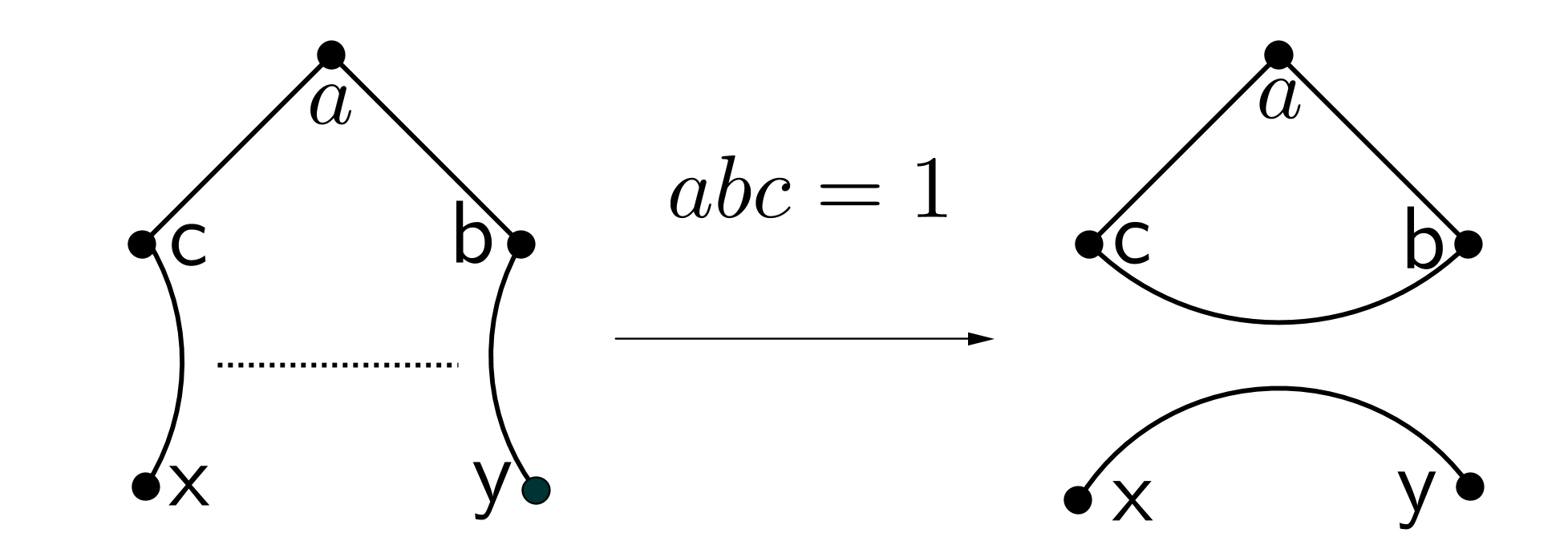} 
	\caption{\textit{Diagram showing bridge-move.}}
	\label{bridge}
\end{figure}
\medskip
We say that a picture $\Gamma$ is \textit{reduced} if it cannot be altered by bridge moves to a picture with a pair of cancelling vertices. 
A picture $\Gamma$ on $D^2$ is  \textit{minimal} if it is non-empty and has the minimum number of vertices 
amongst all pictures over $G$ with the same boundary label as $\Gamma$.
Clearly  minimal pictures are reduced.
Any cyclically reduced word in $G_1 * G_2$ representing the identity element of $G$ occurs as the boundary label of some reduced picture on $D^2$.

\begin{defn}\label{rmk}
	Let $\Gamma$ be 
	a picture  over $G$. 
	Two arcs of $\Gamma$ are said to be \textit{parallel} if they are the only two arcs in the boundary of some simply-connected region $\bigtriangleup$ of $\Gamma$. 
\end{defn}
We will also use the term \textit{parallel} to denote the equivalence relation generated by this relation, and refer to any of the corresponding equivalence classes as a \textit{class of $\omega$ parallel arcs} or \textit{$\omega$-zone}. Given a \textit{$\omega$-zone} 
with $\omega>1$ 
joining vertices $u$ and $v$ of $\Gamma$, consider the $\omega- 1$ two-sided regions separating these arcs. Each such region has a corner label $x_{_u}$ at $u$ and a corner label $x_{_v}$ at $v$, and the picture axioms imply that $x_{_u}x_{_v} = 1$ in $G_1$ or $G_2$. The $\omega -1$ corner labels at $v$ spell a cyclic subword $s$ of length $\omega-1$ of the label of $v$. Similarly the corner labels at $u$ spell out a cyclic subword $t$ of length $\omega -1$ of the label of $u$. Moreover, $s=t^{-1}$. If we assume that $\Gamma$ is reduced, then $u$ and $v$ do not cancel. 
In the spirit of small-cancellation-theory, we refer to 
$t$ and $s$ 
as 
\textit{pieces}. 

\medskip

As in graphs, the \textit{degree} of a vertex in $\Gamma$ is the number of \textit{zones} incident on it. For a region, the \textit{degree} is the number corners it has. We say that a vertex $v$ of $\Gamma$ satisfies the (local) $C(m)$ condition if it is joined to at least $m$ \textit{zones}. We say that $\Gamma$ satisfies $C(m)$ if every interior vertex satisfies $C(m)$.

\section{TECHNICAL RESULTS}
		
		In this section we give a number of results on the structure of $R$ when $D_2(\textbf{S}_{R})\leq 2$, from which Theorem \ref{the1} follows. It is assumed throughout that no element of $\textbf{S}_{R}$ has the form $UV$, where $U$ and $V$ are both uniquely positioned. In particular if $D(a)\geq 2$, there exists at most one $i\in \{1,2,\cdots,D(a)\}$ such that $M_i$ uniquely positioned in the decomposition $R=\prod_{i=1}^{D(a)}aM_i$. We begin with the proof of part(a) of Theorem \ref{the1}.

	\begin{lem}\label{lems1}
		If $D_2(\textbf{S}_{R})=1$, then $R$ has a cyclic conjugate of the form $aM$ or $aXbX^{-1}$, where $a,b$ are letters of order $2$ and $M$ does not involve any letter of order $2$.
	\end{lem}
	
	\begin{proof}
	Since $D_2(\textbf{S}_{R})=1$, we can assume without loss of generality that $R=aM$, where $M$ is a word in $G_1*G_2$ which does not involve $a$. We now proceed to show that either $M$ does not involve any  letter of order $2$ or $M$ can be decomposed in the form $XbX^{-1}$, where $b\in G_1\cup G_2$  is a letter of order $2$ and $X$ is a (possibly empty) word in $G_1*G_2$.
	
	\medskip
	Suppose by contradiction that $M$ has a decomposition of the form $XbY$ with $b^2=1$ and $X\neq Y^{-1}$. Note that we can assume without loss of generality that $0<\ell(X)<\ell(Y)$. Clearly, if $\ell(X)=\ell(Y)>0$, then both $aX$ and $bY$ are uniquely positioned which is a contradiction. There is nothing to prove if $\ell(X)=\ell(Y)=0$. Also if $\ell(X)= 0\neq \ell(Y)$, we get a contradiction since $ab$ and $Y$ will be uniquely positioned. Hence the inequality $0<\ell(X)<\ell(Y)$ holds.
	
	\medskip
	Suppose that $X^2=1$ and $Y^2=1$ holds simultaneously. Then by setting $X=X_1pX_1^{-1}$ and $Y=Y_1^{-1}qY_1$, where $X_1,Y_1\in G_1*G_2$ and $p,q$ are distinct letters of order $2$ in $G_1\cup G_2$, we can replace $R$ with
	$$R'=pX'qY',$$
	where $X'=(Y_1bX_1)^{-1}$ and $Y'=Y_1aX_1$. Since $a\neq b$, we have that $X'\neq Y'^{-1}$. Given  that $\ell(X')=\ell(Y')$, we easily conclude that $pX'$ and $qY'$ are uniquely positioned. This is a contradiction.	
	
\medskip
Suppose that  $X^2=1\neq Y^2$. By the assumption that $D_2(\textbf{S}_{R})=1$, we know that $X$ can not be equal to a segment of $Y$. Hence $aX$ and $bY$ are both uniquely positioned. This is a contradiction. Similarly, suppose that $X^2\neq 1=Y^2$.  Since $\ell(X)<\ell(Y)$ and  $D_2(\textbf{S}_{R})=1$, we have that both $bY$ and $Ya$ are uniquely positioned. Hence, neither $aX$ nor $Xb$ is uniquely positioned. This means that $X^{-1}$ is identically equal to an initial and a terminal segment of $Y$. Therefore, $X^2=1$. This is a contradiction. 

\medskip
Finally if $X^2\neq 1\neq Y^2$, then $aXb$ and $Y$ are both uniquely positioned. This contradiction completes the proof.
\end{proof}

	\begin{lem}\label{lems2}
	Suppose that $D_2(\textbf{S}_{R})=2$. Then $R$ has a cyclic conjugate of the form $aXbX^{-1}$ where $a$ is a letter of order $2$.
	\end{lem}
	
	\begin{proof}
		Since $D_2(\textbf{S}_{R})=2$, we can assume without loss of generality that $$R = aM_1aM_2,$$ where $M_1,M_2\in G_1*G_2$, and neither involves the letter $a$. By assumption $M_1$ and $M_2$ can not be uniquely positioned simultaneously. If  $M_1^2=1$ and $M_2^2=1$ hold simultaneously, then by replacing $R$ with a cyclic conjugate, it can be shown that $R$ has the desired form.  Without loss of generality, we can assume that $1\leq \ell(M_1)\leq \ell(M_2)$. 
		
		\medskip
		Suppose that $\ell(M_1)=\ell(M_2)$. We can not have $M_1= M_2$ since $R$ is not a proper power. Also if $M_1= M_2^{-1}$, then there is nothing to prove. So we assume without loss of generality that $M_1^2=1$ and $ M_2$ is uniquely positioned. If $\ell(M_1)=1$, then there is nothing to prove since $M_1$ has order $2$ and so $R$ has the desired form. Hence we assume that $\ell(M_1)= \ell(M_2)\geq 3$. Let $M_1=XpX^{-1}$ and $M_2=YqZ$, with  $p,q\in G_1\cup G_2$, $p^2=1$, $\ell(Y)=\ell(Z)$ and $Y\neq Z^{-1}$ (as otherwise there is nothing to prove). Then
		$$R=aXpX^{-1}aYqZ.$$		
		Set 
		 \begin{align*}
		 	U= & aYq, \ \  \qquad\qquad U'=qZa,\\
		 	V= & ZaXpX^{-1},\qquad V'=XpX^{-1}aY.
		 \end{align*}
		 Clearly, $V^2\neq 1\neq V'^2$ since $D(a)=2$. Also since $Y\neq Z^{-1}$, it follows that $V$ and $V'$ are simultaneously uniquely positioned. Hence neither $U$ nor $U'$ is uniquely positioned. It is easy to see that this means that $U^2=1$ or $U'^2=1$ or $U'=U^{\pm 1}$. However, any such occurrence will imply that $a=q$ or $Y=Z^{-1}$. This is a contradiction.
			
\medskip		
Now suppose that $\ell(M_i)\neq \ell(M_j)$, where  $i,j\in \lbrace 1,2\rbrace$ with $i\neq j$. Note that it is not possible to have $M_i^2\neq 1$ and $M_j^2\neq 1$ holding simultaneously since that will imply that $aM_ia$ and $M_j$ are both uniquely positioned, assuming $\ell(M_i)<\ell(M_j)$. Suppose that $M_i^2=1$. Let $M_i=XpX^{-1}$ and $M_j=YqZ$, with $p,q\in G_1\cup G_2$, $p^2=1$, $\ell(Y)=\ell(Z)$ and $Y\neq Z^{-1}$. We claim that exactly one of $aY$ or $Za$ is uniquely positioned. This is because if both are uniquely positioned, then there is nothing to prove. Also if neither is uniquely positioned, then $Y=Z^{-1}$. In both cases we get a contradiction.

\medskip
By symmetry we assume that $aY$ is uniquely positioned, and hence $qZaM_i$ is not. This leads to a contradiction when $\ell(Y)\geq \ell(M_i)$ since that will mean $Y=Z^{-1}$. Suppose then that $\ell(Y)< \ell(M_i)$. This implies that $M_i$ is an initial or terminal segment of $M_j$. Hence, we have that $M_j=M_iW$ or $M_j=WM_i$ for some $W\in G_1*G_2$, depending on whether $M_i$ is an initial or terminal segment of $M_j$. Note that $\ell(W)=2n$ for some integer $n>0$. Now we replace $R$ by 
$$R'= pMpN,$$
where $M=X^{-1}aX$ and $N=X^{-1}WaX$ or $N=X^{-1}aWX$. We consider first the case when $N=X^{-1}WaX$. In this case, the initial segment $X^{-1}W$ of $N$ has length $\ell(X^{-1}W)\geq \ell(X)+2$. Since $D_2(\textbf{S}_{R})=2$, $X^{-1}W$ neither involves $a$ nor $p$. It follows that $aXpXaX^{-1}p$ is uniquely positioned. Hence, $X^{-1}W$ is not uniquely positioned. The length condition on $X^{-1}W$ implies that $(X^{-1}W)^2=1$. Again since $D_2(\textbf{S}_{R})=2$, $X$ does not involve a letter of order $2$. So $W=SxS^{-1}X$, for some (possibly empty) word $S$ and some letter $x$ of order $2$. Hence $$R'= pX^{-1}aXpX^{-1}SxS^{-1}XaX.$$
Consider the cyclic subwords $W_1=S^{-1}XaXpX^{-1}aX$ and $W_2=pX^{-1}Sx$.
Clearly, $W_1^2\neq 1$ as otherwise $S$ is empty and more importantly $X^2=1$, which is a contradiction.  Also, $W_2^2\neq 1$ since $p\neq x$. In fact, it is easy to see that both $W_1$ and $W_2$ are uniquely positioned. This is a contradiction. Similar argument works when $N=X^{-1}aWX$ by replacing  $W_1$ and $W_2$ with their inverses. This completes the proof.
\end{proof}

 By combining Lemmas \ref{lems1} and \ref{lems2}, we obtain Theorem \ref{the1}.

\medskip
The remaining results in this section are consequences of results about a picture $\Gamma$ over $G$. First, we give a necessary and sufficient condition under which the word $R$ has a decomposition into a pair of uniquely positioned subwords when $D_2(\textbf{S}_{R})=1$.

\begin{lem}\label{lems4}
	Let $r$ be a cyclically reduced word which is not a proper power in the free product $G_1*G_2$ such that $D_2(\textbf{S}_{r})=1$. Then, $r$ has a decomposition into two uniquely positioned subwords if and only if $\ell(r)>2$ and there exists $r'\in \textbf{S}_{r}$ such that $r'=aXxYyX^{-1}$ with $X,Y,x,y,a\in G_1*G_2$, $\ell(Y)\geq 1$, $\ell(x)=\ell(y)=\ell(a)=1$, $x\neq y^{-1}$ and $a^2=1$.
\end{lem}
\begin{proof}
	Suppose that $r$ has a decomposition into two uniquely positioned subwords $U$ and $V$. Since $D(\textbf{S}_{r})=1$, we have that $\ell(r)>2$. Without loss of generality, it follows that a cyclic conjugate of $r$ has the form $$r'=aU_2VU_1,$$
	where $U=U_1aU_2$ and $a^2=1$. Hence $U_2VU_1=XYX^{-1}$ for some words $X,Y\in G_1*G_2$, where $X$ is possibly empty. Since $U$ and $V$ are uniquely positioned in $r$, we conclude that $\ell(Y)\geq 3$ and the first and last letters of $Y$ are not inverses. The result follows.
	
	\medskip
	For the other direction, suppose $r'=aXxYyX^{-1}$ with $X,Y,x,y,a\in G_1*G_2$, $\ell(x)=\ell(y)=\ell(a)=1$, $x\neq y^{-1}$ and $a^2=1$. Then $aXx$ is clearly uniquely positioned in $r$ since $x\neq y^{-1}$. For the same reason, we deduce from part(a) of Theorem \ref{the1} that $XxYyX^{-1}$ has no element of order $2$. In particular, this means that $YyX^{-1}$ and its inverse do not intersect (in an initial or terminal segment). We claim that this means that $YyX^{-1}$ is also uniquely positioned. We prove this by contradiction by assuming that $YyX^{-1}$ is not uniquely positioned and showing that  $XxYyX^{-1}$ contains an element of order $2$. 

\medskip
Let $XxYyX^{-1}=x_1x_2\cdots x_n,$ with $X=x_1x_2\cdots x_p$. Suppose that $YyX^{-1}$ is not uniquely positioned. Then, $(YyX^{-1})^{\pm 1}$ is identically equal to some segment of $XxYyX^{-1}$. This segment must intersect $YyX^{-1}$. By above discussion, we have that $YyX^{-1}$ is identically equal to the segment $$x_kx_{k+1}\cdots x_{\ell(YyX^{-1})-1},$$ with $k\leq p$. Hence, we have that the terminal segment of $XxYyX^{-1}$ of length $n+1-k$ has period $\lambda=p+2-k$. Consider the initial segment of this periodic segement given by
$$W_k=x_kx_{k+1}\cdots x_{n+k-(p+2)}.$$
In particular $W_k$ is of length $n-(p+1)$. Note that $X^{-1}=x^{-1}_px^{-1}_{p-1}\cdots x^{-1}_1=x_{n+1-p}x_{n+2-p}\cdots x_n$. If $x_i=x^{-1}_i$ for some $k\leq i\leq p$, then we are done. Suppose not. If $x_p$ (alternatively $x_k$) is identified with $x^{-1}_i$ for some $k\leq i\leq p$, then $x_{\frac{p+i}{2}}=x^{-1}_{\frac{p+i}{2}}$ (alternatively $x_{\frac{k+i}{2}}=x^{-1}_{\frac{k+i}{2}}$). This is a contradiction. Otherwise,  both $x_k$ and $x_p$ are identified with $x^{-1}_i$ and $x^{-1}_j$ respectively, where $1\leq j\leq i<k-1$ (since we are in a free product). In fact, $j=i+k-p<2k-1-p$.
Choose $j$ such that under this periodicity, $x^{-1}_j$ is the letter that provides the first identification with $x_p$. We claim that $j+\lambda$ lies between $k$ and $p$. To verify this claim, it is enough to show that $p\geq j+\lambda$. We have that
$j+\lambda<2k-1-p+\lambda=k+1$. Therefore, $j+\lambda\leq k\leq p$.
	Hence $x_{p}=x^{-1}_{j+\lambda}$ and $j+\lambda\leq p$. By the choice of $j$, we must have that $k\leq j+\lambda\leq p$.
	This is a contradiction. Hence $YyX^{-1}$ is uniquely positioned. This completes the proof.
	
 \end{proof}

\begin{lem}\label{lems3}
	Let $\Gamma$ be a reduced picture over $G$ on $D^2$ where $R=aXbX^{-1}$ and $a^2=b^2=1$. Then either $\Gamma$ is empty or  $\Gamma$ satisfies   $C(6)$.
\end{lem}

\begin{proof}
	Suppose $\Gamma$ a non-empty picture over $G$ which is reduced. Suppose also that $\Gamma$ contains some interior vertex $v$ of degree less than $6$. Then $v$ is connected to another vertex $u$ by a zone containing $a$ or $b$. Using this zone, we can do bridge-moves so that $u$ and $v$ form a dipole. This contradicts the assumption that $\Gamma$ was reduced.
\end{proof}

\medskip
We obtain from Lemmas \ref{lems1}, \ref{lems4} and \ref{lems3} the following corollary.
\begin{cor}\label{cor2}
Let $D_2(\textbf{S}_{R})\leq 2$ and $R\neq aXbX^{-1}$ with $a^2=1\neq b^2$. Then any non-empty reduced picture on $D^2$ over $G$ satisfies  $C(6)$.
\end{cor}
\begin{proof}
By Lemma \ref{lems4}, either $R$ has a decomposition $UV$ with $U,V$ uniquely positioned in $R$ or $R$ has the form (up to cyclic conjugation) $aXbX^{-1}$ with $a^2=b^2=1$. For the first case, the proof is exactly as it is in [\cite{dh} Lemma 3.1]. For the latter case, the result follows from Lemmas \ref{lems3}.
\end{proof}

\section{APPLICATIONS}
In this section we deduce a number of applications of Theorem \ref{the2}. But first, we recall the setting. 

\medskip
Let $G_1$ and $G_2$ be non-trivial groups and $R\in G_1*G_2$ which is not a proper power and has length at least $2$. We also require that no letter  of order $2$ involved in $R$ appears more than twice i.e. $D_2(\textbf{S}_{R})\leq 2.$
For a natural number $m\geq 3$, $G$ is the quotient of $G_1*G_2$ by the normal closure of $R^m$.

\begin{pfof}{Theorem \ref{the2}}
	This follows from Part (b) or Theorem \ref{the1} and Corollary \ref{cor2}.
\end{pfof}

When $R$ has a conjugate of the form $aXbX^{-1}$ and $a^2=1\neq b^2$, we will call $R$ \textit{exceptional}. As mentioned earlier, there are results in the literature on the two classes and we list them without proof. We begin the non-exceptional case. For this case the proofs can be found in \cite{dh}.

\begin{thm}\label{the3} Suppose that $G$ is as above and $R$ is not exceptional. Then the following hold.
		\begin{itemize}
			\item[(a)] \textbf{Freiheitssatz.}
					The natural homomorphisms $G_1\rightarrow G$ and $G_2\rightarrow G$ are injective.
					\item[(b)] \textbf{Weinbaum's Theorem.}
					No non-empty proper subword of $R^m$ represents the identity element of $G$.
			\item[(c)] \textbf{Word problem.} If $G_1$ and $G_2$ are given by a recursive presentation with soluble word problem, then so is $G$. Moreover, the generalized word problem for $G_1$ and $G_2$ in $G$ is soluble with respect to these presentations.
			\item[(d)] \textbf{The Identity Theorem.} 
			$N(R^m)/[N(R^m), N(R^m)] = \mathbb{Z}G/(1-R)\mathbb{Z}G$
			as a (right) $\mathbb{Z}G$-module.
		\end{itemize}	
	\end{thm}
	
	\begin{cor}
		There are natural isomorphisms for all $q > 3$;
		\begin{align*}
			H^{q}(G ; -) \longrightarrow    H^{q}(G_1 ; -) & \times H^{q}(G_2 ; -) \times H^{q}(\mathbb{Z}_m ; -),\\
			H_{q}(G ; -) \longleftarrow    H_{q}(G_1 ; -) & \oplus H^{q}(G_2 ; -) \oplus H^{q}(\mathbb{Z}_m ; -);
		\end{align*}
		a natural epimorphism
		\begin{align*}
			H^{2}(G ; -) \longrightarrow    H^{2}(G_1 ; -) & \times H^{2}(G_2 ; -) \times H^{2}(\mathbb{Z}_m ; -),
		\end{align*}
		and a natural monomorphism
		\begin{align*}
			H_{2}(G ; -) \longleftarrow    H_{2}(G_1 ; -) & \oplus H_{2}(G_2 ; -) \oplus H_{2}(\mathbb{Z}_m ; -).
		\end{align*}
	\end{cor}

These are defined on the category of $\mathbb{Z}G$-modules, $\mathbb{Z}_m$ is the cyclic subgroup of order $m$ generated by $R$, and all these maps are induced by restriction on each
factor.

\medskip
Next we consider the exceptional case. In this case, $G$ is  called a one-relator product induced from the generalized triangle group $H$, described as follows. Let $A:=\<a\>$ and $X^{-1}BX:=\<b\>$ be the cyclic subgroups of $G_1$ or $G_2$ generated by $a$ and $b$ respectively. Then $H:=(A*B)/N(R^m)$. Note that $G$ can be realized as
a push-out of groups as shown in Figure \ref{push-out}. 

\begin{figure}[ht]
	\begin{tikzpicture}[>=latex]
	\centering
	\node (w) at (0,0) {\(A*B\)};
	\node (x) at (0,-2) {\(G_1 * G_2\)};
	\node (y) at (2.4,0) {\(H\)};
	\node (z) at (2.4,-2) {\(G\)};
	\draw[->] (w) -- (y);
	\draw[->] (w) -- (x);
	\draw[->] (x) -- (z);
	\draw[->] (y) -- (z);
	\end{tikzpicture}
	\caption{\textit{Push-out diagram.}}\label{push-out}
\end{figure}

This pushout representation of $G$ is referred to  as a generalized triangle group description of $G$,
and we require it to be \textit{maximal} in the sense \cite{Ih1}. Another technical requirement is  that $(a, b)$ be \textit{admissible}: whenever both $a$ and $b$ belong to same
factor, say $G_1$, then either the subgroup of $G_1$ generated by $\{a,b\}$ is cyclic or  $\<a\>\cap\<b\>=1$. It is very easy to verify that these conditions are satisfied in our setting. Hence the results in \cite{hs} hold.

\begin{thm}\label{the11} \label{the4} Suppose that $G$ is as above and $R$ is exceptional. Then the following hold.
		\begin{itemize}
		\item[(a)]\textbf{Freiheitssatz.} The natural homomorphisms $G_1\rightarrow G$ and $G_2\rightarrow G$ are injective.
			\item[(b)] \textbf{Weinbaum's Theorem.}
			No non-empty proper subword of $R^m$ represents the identity element of $G$.
			\item[(c)] \textbf{Membership problem.} Assume that the membership problems for $\langle a\rangle$ and $\langle b \rangle$ in $G_1*G_2$ are
			solvable.  Then the  word problem for $G$ is also soluble.
			\item[(d)] \textbf{Mayer-Vietoris.} The pushout of groups in Figure \ref{push-out}  is {\em geometrically
				Mayer-Vietoris} in the sense of \cite{hs}.  In particular it
			gives rise to Mayer-Vietoris sequences
			$$\cdots \to H_{k+1}(G,M)\to H_k(A*B,M)\to$$ $$ H_k(G_1*G_2,M)\oplus H_k(H,M)\to H_k(G,M)\to\cdots  $$
			and $$\cdots \to H^k(G,M)\to H^k(G_1*G_2,M)\oplus H^k(H,M)$$
			$$\to H^k(A*B,M)\to H^{k+1}(G,M)\to\cdots $$ for any $\mathbb{Z}G$-module $M$.	
		\end{itemize}
\end{thm}

\end{document}